\documentclass[12pt]{amsart}
\usepackage{epsfig,url,youngtab}
\usepackage{mathdots}
\usepackage[margin=1in]{geometry}

\newtheorem{theorem}{Theorem}[section]
\newtheorem{lemma}[theorem]{Lemma}
\newtheorem{conjecture}[theorem]{Conjecture}

\newtheorem{corollary}[theorem]{Corollary}

\theoremstyle{definition}

\theoremstyle{remark}

\begin{document}

\title[]
{Congruences for\\
the Almkvist-Zudilin numbers}



\author{Tewodros Amdeberhan}
\address{Department of Mathematics,
Tulane University, New Orleans, LA 70118}
\email{tamdeber@tulane.edu}

\subjclass[2010]{Primary ??}

\date{\today}


\begin{abstract}
Given a prime number $p$, the study of divisibility properties of a sequence $c(n)$ has two contending approaches: $p$-adic valuations and superconcongruences. The former searches for the highest power of $p$ dividing $c(n)$, for each $n$; while the latter (essentially) focuses on the maximal powers $r$ and $t$ such that $c(p^rn)$ is congruent to $c(p^{r-1}n)$ modulo $p^t$. This is called supercongruence. In this note, we prove modest supercongruences for certain sequences that have come to be known as the Almkvist-Zudilin numbers and two other naturally related ones. 
\end{abstract}

\maketitle

\newcommand{\ba}{\begin{eqnarray}}
\newcommand{\ea}{\end{eqnarray}}
\newcommand{\ift}{\int_{0}^{\infty}}
\newcommand{\nn}{\nonumber}
\newcommand{\no}{\noindent}
\newcommand{\lf}{\left\lfloor}
\newcommand{\rf}{\right\rfloor}
\newcommand{\realpart}{\mathop{\rm Re}\nolimits}
\newcommand{\imagpart}{\mathop{\rm Im}\nolimits}

\newcommand{\op}[1]{\ensuremath{\operatorname{#1}}}
\newcommand{\pFq}[5]{\ensuremath{{}_{#1}F_{#2} \left( \genfrac{}{}{0pt}{}{#3}
{#4} \bigg| {#5} \right)}}

\newtheorem{Definition}{\bf Definition}[section]
\newtheorem{Thm}[Definition]{\bf Theorem}
\newtheorem{Example}[Definition]{\bf Example}
\newtheorem{Lem}[Definition]{\bf Lemma}
\newtheorem{Cor}[Definition]{\bf Corollary}
\newtheorem{Prop}[Definition]{\bf Proposition}
\numberwithin{equation}{section}

\section{Introduction}

\noindent
Let us fix some notational conventions. Denote the set of positive integers by $\mathbb{P}$. For $m\in\mathbb{P}$, let $\equiv_m$ represent congruence modulo $m$. Throughout, assume $p\geq5$ is a prime.

\noindent
The {\em Ap\'ery numbers} $A(n)=\sum_{k=0}^n\binom{n}k^2\binom{n+k}k^2$ were valuable to R. Ap\'ery in his celebrated proof ~\cite{A} 
of the irrationality of $\zeta(3)$. Since then these numbers have been a subject of much research. For example, they stand among a host of other sequences with the property 
$$A(p^rn)\equiv_{p^{3r}} A(p^{r-1}n)$$
now known as {\em supercongruence} $-$ a term dubbed by F. Beukers ~\cite{B}.

\noindent
At the heart of many of these congruences sits the classical example $\binom{pb}{pc}\equiv_{p^3}\binom{b}c$ which is a stronger variant of the famous Lucas's congruence $\binom{pb}{pc}\equiv_p\binom{b}c$. For a compendium of references on the subject of Ap\'ery-type sequences, see ~\cite{S}.

\noindent
In this paper, true to tradition, we shall investigate similar type of divisibility properties (i.e. supercongruences) of the following three sequences. For  $i\in\{0,1,2\}$ and $n\in\mathbb{P}$, define
\begin{align*} a_i(n):&=\sum_{k=0}^{\lfloor(n-i)/3\rfloor}(-1)^{n-k}\binom{3k+i}k\binom{2k+i}k\binom{n}{3k+i}\binom{n+k}k3^{n-3k-i} \\
\end{align*}
In recent literature, $a_0(n)$ are referred to as the Almkvist-Zudilin numbers.
Our motivation for the present work here emanates from the following claim found in ~\cite{O} (see also ~\cite{CCS}).

\begin{conjecture} For a prime $p$ and $n\in\mathbb{P}$, the Almkvist-Zudilin numbers satisfy
$$a_0(pn)\equiv_{p^3}a_0(n).$$
\end{conjecture}

\smallskip
\noindent
Our main results can be summarized as: 

{\em if $p$ is a prime and $n\in\mathbb{P}$, then $a_0(pn)\equiv_{p^3}a_0(n)$ and $a_1(pn)\equiv_{p^2}a_2(pn)\equiv_{p^2}0$.}

\section{Preliminary results}

\begin{lemma}  Let $f(k)=\binom{3k}{k,k,k}\frac{3^{-3k}}{k+1}$ and $g(k)=\binom{3k+1}{k,k,k+1}\frac{3^{-3k}}{k+2}$. We have

\smallskip
(a) $\sum_{j\geq0}\binom{b}j\binom{c}j=\binom{b+c}c$,

\smallskip
(b) $\sum_{k=0}^mf(k)=9\binom{m+2}2f(m+1)$,

\smallskip
(c) $\sum_{k=0}^mg(k)=\frac{9(m+1)(m+3)}8g(m+1)$. 
\end{lemma}
\begin{proof} (a) is Vandermode-Chu's identity; (b) is due to $\frac13\binom{m+2}2f(m+1)-\frac13\binom{m+1}2f(m)=f(m)$;
(c) is verified by $\frac{9(m+1)(m+3)}8g(m+1)-\frac{9m(m+2)}8g(m)=g(m)$.  \end{proof}

\begin{corollary} For each prime $p$, we have the following congruences:
$$\sum_{k=0}^{\lfloor(p-1)/3\rfloor}\binom{3k}{k,k,k}\frac{3^{-3k}}{k+1}\,\equiv_p\,\,0 \qquad \text{and}
\qquad \sum_{k=0}^{\lfloor(p-2)/3\rfloor}\binom{3k+1}{k,k,k+1}\frac{3^{-3k}}{k+2}\,\equiv_p\,\,0.$$
\end{corollary}
\begin{proof} By Lemma 2.1(b) we gather $\sum_{k=0}^m\binom{3k}{k,k,k}\frac{3^{-3k}}{k+1}=
\binom{m+2}2\binom{3m+4}{m+1}\binom{2m+3}{m+1}\frac{3^{-3m-1}}{3m+4}$. Put $m=\lfloor\frac{p-1}3\rfloor$.
If $p=3\ell+1\geq11$ then $\binom{3m+4}{m+1}=\binom{p+3}{\ell+1}\equiv_p\binom{3}{\ell+1}\equiv_p0$; if $p=3\ell+2\geq11$ then $\binom{3m+4}{m+1}=\binom{p+2}{\ell+1}\equiv_p\binom{2}{\ell+1}\equiv_p0$. For $p=5, 7$, we have $\binom{2m+3}{m+1}\equiv_p0$. In all cases, the desired congruence holds.

\noindent
For the second congruence, Lemma 2.1(c) gives $\sum_{k=0}^m\binom{3k+1}{k,k,k+1}\frac{3^{-3k}}{k+2}=\frac{m+1}8
\binom{3m+4}{m+1}\binom{2m+3}{m+1}3^{-3m-1}$. Put $m=\lfloor\frac{p-2}3\rfloor$.
If $p=3\ell+1\geq11$ then $\binom{3m+4}{m+1}=\binom{p+3}{\ell+1}\equiv_p\binom{3}{\ell+1}\equiv_p0$; if $p=3\ell+2\geq11$ then $\binom{3m+4}{m+1}=\binom{p+2}{\ell+1}\equiv_p\binom{2}{\ell+1}\equiv_p0$. For $p=5, 7$, we have $\binom{2m+3}{m+1}\equiv_p0$. In all cases, the required congruence remains valid.
\end{proof}

\noindent
{\em Fermat quotients} are numbers of the form $b_p(x)=\frac{x^{p-1}-1}p$ and they played a useful role in the study
of cyclotomic fields and Fermat's Last Theorem, see ~\cite{R}. We need one of their divisibility properties as stated in the
next result, see Lehmer (\cite{L}, page 358).
\begin{lemma} For each prime $p$, it holds that
$$\frac{b_p(3)}2-\frac{p\,b_p(3)^2}4\equiv_{p^2}\sum_{r=1}^{\lfloor p/3\rfloor}\frac1{p-3r}.$$
\end{lemma}

\begin{corollary} Let $H(n)=\sum_{j=1}^n\frac1j$ be the harmonic numbers, $p$ a prime and $n\in\mathbb{P}$, Then,
$$3\left(\frac{3^{(p-1)n}-1}p\right)\equiv_p-(2n)\,H(\lfloor p/3\rfloor).$$
\end{corollary}
\begin{proof} Lemma 2.3 implies $3b_p(3)\equiv_p-2H(\lfloor p/3\rfloor)$. On the other hand, Fermat's little theorem implies $b_p(x^n)\equiv_pn\,b_p(x)$. Taking $x=3$, the proof follows.
\end{proof}

\begin{lemma} For $n\in\mathbb{P}$, we have the identity
$$\sum_{k=1}^n(-1)^k\binom{n}k\binom{n+k}k\frac1k=-2H(n).$$
\end{lemma}
\begin{proof} Given any $y\in\mathbb{P}$, E. Mortenson (\cite{M}, page 990) made application of the Wilf-Zeilberger method to prove the identity
$$\sum_{k=0}^n(-1)^k\binom{n}k\binom{n+k}k\frac1{k+y}=\frac{(-1)^n}y\prod_{j=1}^n\frac{y-j}{y+j}.$$
Actually, the same method offers that the identity is valid for an indeterminate $y$. Now, subtract $\frac1y$ from both sides and take the limit as $y\rightarrow0$. The right-hand side takes the form
$$\frac1{n!}\lim_{y\rightarrow0}\left[\frac{\prod_{j=1}^n(j-y)-\prod_{j=1}^n(j+y)}y\right]=-2\sum_{k=1}^n\frac1k.$$
The conclusion is clear.
\end{proof}

\begin{lemma} Suppose $p$ is a prime and $z$ is a variable. Then,
$$\sum_{k=1}^{\lfloor p/3\rfloor}(-1)^k\binom{\lfloor p/3\rfloor}k\binom{\lfloor p/3\rfloor+k}k\frac{z^k}k
\equiv_p\sum_{k=1}^{\lfloor p/3\rfloor}\binom{3k}{k,k,k}\frac{3^{-3k}z^k}k.$$
\end{lemma}
\begin{proof} The congruence actually holds term-by-term and that is how we proceed. Also, observe that $\binom{n}k\binom{n+k}k=\binom{2k}k\binom{n+k}{2k}$. If $p\equiv_31$, then $\lfloor\frac{p}3\rfloor=\frac{p-1}3$
and hence
\begin{align*} \binom{\frac{p-1}3+k}{2k}
&=\frac{\frac{p-1}3(\frac{p-1}3+k)}{(2k)!}\prod_{j=1}^{k-1}\left(\frac{p-1}3\pm j\right) \\
&\equiv_p\frac{(-1)^k(3k-1)}{3^{2k}(2k)!}\prod_{j=1}^{k-1}(3j\pm1) 
=\frac{(-1)^k(3k)!}{3^{3k}(2k)!k!}. \end{align*}
Therefore, we gather that
$$(-1)^k\binom{\frac{p-1}3}k\binom{\frac{p-1}3+k}k=
(-1)^k\binom{2k}k\binom{\frac{p-1}3+k}{2k}\equiv_p\frac{(3k)!}{3^{3k}!k!^3}=\binom{3k}{k,k,k}3^{-3k}.$$
The proof follows by summing over $k$. The case $p\equiv_3-1$ runs analogously. \end{proof}

\begin{corollary} For a prime $p$, we have the congruence
$$pn\sum_{k=1}^{\lfloor p/3\rfloor}\binom{3k}{k,k,k}\frac{3^{-3k}}k\equiv_{p^2}3\left(3^{(p-1)n}-1\right).$$
\end{corollary}
\begin{proof} Follows from combining Corollary 2.4, Lemma 2.5 and Lemma 2.6 with $z=1$.
\end{proof}

\section{Main results on the sequences $a_1(n)$ and $a_2(n)$}

\begin{theorem} For a prime $p$ and $n\in\mathbb{P}$, we have $a_1(pn)\equiv_{p^2} 0$.
\end{theorem}
\begin{proof} Let $k=pm+r$ for $0\leq r\leq p-1$. Note: $3k+1\leq pn$ and $3pm+3r+1\leq pn$. Write
\begin{align*} a_1(pn)=\sum_{m=0}^{\lfloor n/3\rfloor}\sum_{r=0}^{p-1}
(-1)^{pn-pm-r}&\binom{3pm+3r+1}{pm+r}\binom{2pm+2r+1}{pm+r} \times \\
&\binom{pn}{3pm+3r+1}\binom{pn+pm+r}{pm+r}3^{pn-3pm-3r-1}.\end{align*}
If $t:=3r+1\geq p+1$, it is easy to show that the following terms vanish modulo $p^2$:
$$\binom{3pm+t}{pm+r}\binom{2pm+2r+1}{pm+r}\binom{pn}{3pm+t}=
\binom{3pm+t}{pm+r,pm+r,pm+r+1}\binom{pn}{3pm+t}.$$
Therefore, we may restrict to the remaining sum with $3r+1\leq p$:
\begin{align*} a_1(pn)=\sum_{m=0}^{\lfloor n/3\rfloor}\sum_{r=0}^{\lfloor(p-1)/3\rfloor}
(-1)^{pn-pm-r}&\binom{3pm+3r+1}{pm+r}\binom{2pm+2r+1}{pm+r} \times \\
&\binom{pn}{3pm+3r+1}\binom{pn+pm+r}{pm+r}3^{pn-3pm-3r-1}.\end{align*}
We need Lucas's congruence $\binom{pb+c}{pd+e}\equiv_p\binom{d}d\binom{c}e$ to arrive at 
\begin{align*} a_1(pn)\equiv_p\sum_{m=0}^{\lfloor n/3\rfloor}\sum_{r=0}^{\lfloor(p-1)/3\rfloor}
(-1)^{pn-pm-r}&\binom{3m}m\binom{3r+1}r\binom{2m}m\binom{2r+1}r\times \\
&\binom{pn}{3pm+3r+1}\binom{n+m}m3^{pn-3pm-3r-1}.\end{align*}
For $0<i<p$, we apply Gessel's congruence $\binom{p}i\equiv_{p^2}(-1)^{i-1}\frac{p}i$ (if $p=3r+1$, know that $k$ is even; in this case, still the corresponding term properly absorbs into the sum below) so that
\begin{align*} \binom{pn}{3pm+3r+1}&=\frac{pn}{3pm+3r+1}\binom{pn-1}{3pm+3r} \\
&=\frac{pn}{3pm+3r+1}\binom{p(n-1)+p-1}{3pm+3r} 
\equiv_{p^2}(-1)^r\frac{pn}{3r+1}\binom{n-1}{3m}, \end{align*}
which leads to
\begin{align*} a_1(pn)\equiv_{p^2}pn\sum_{m=0}^{\lfloor n/3\rfloor}\sum_{r=0}^{\lfloor(p-1)/3\rfloor}
(-1)^{n-m-r}&\binom{3m}m\binom{3r+1}r\binom{2m}m\binom{2r+1}r\times \\
&\frac{(-1)^r}{3r+1}\binom{n-1}{3m}\binom{n+m}m3^{pn-3pm-3r-1}.\end{align*}
Next, we use Fermat's Little Theorem and {\em decouple} the double sum to obtain
\begin{align*} a_1(pn)\equiv_{p^2}n&\sum_{m=0}^{\lfloor n/3\rfloor}(-1)^{n-m}3^{n-3m-1}\binom{3m}m\binom{2m}m
\binom{n-1}{3m}\binom{n+m}m \times \\
&p\sum_{r=0}^{\lfloor(p-1)/3\rfloor}\binom{3r+1}r\binom{2r+1}r\frac{3^{-3r}}{3r+1}.\end{align*}
If we invoke Corollary 2.2, the sum over $r$ is already divisible by $p^2$ which is enough to conclude $a_1(pn)\equiv_{p^2}0$.
\end{proof} 

\begin{theorem} For a prime $p$ and $n\in\mathbb{P}$, we have $a_2(pn)\equiv_{p^2} 0$.
\end{theorem}
\begin{proof} The idea is similar to Theorem 0.3, hence it is omitted to avoid needless duplicity.
\end{proof}

\newpage

\section{Main result on the sequence $a_0(n)$}

\noindent
In this section, we prove a weaker version of Conjecture 1.1 which still required a more delicate touch than what has been demonstrated in the previous section for the other two sequences. We also believe that our techniques pave the way in settling the conjecture, fully.

\begin{theorem} For a prime $p$ and $n\in\mathbb{P}$, we have $a_0(pn)\equiv_{p^2} a_0(n)$. 
\end{theorem}
\begin{proof}  Let $k=pm+r$ for $0\leq r< p$. Note: $3k=3pm+3r\leq pn$. Using the new parameters, 
\begin{align*} a_0(pn)=\sum_{m=0}^{\lfloor n/3\rfloor}\sum_{r=0}^{p-1}
(-1)^{pn-pm-r}&\binom{3pm+3r}{pm+r}\binom{2pm+2r}{pm+r} \times \\
&\binom{pn}{3pm+3r}\binom{pn+pm+r}{pm+r}3^{pn-3pm-3r}.\end{align*}
If $3r\geq p+1$, a straight-forward argument shows the following terms vanish modulo $p^2$:
$$\binom{3pm+3r}{pm+r}\binom{2pm+2r}{pm+r}\binom{pn}{3pm+3r}=
\binom{3pm+3r}{pm+r,pm+r,pm+r}\binom{pn}{3pm+3r}.$$
Therefore, we may restrict to the remaining sum with $3r\leq p$: thus $a_0(pn)$ becomes
$$\sum_{m=0}^{\lfloor n/3\rfloor}\sum_{r=0}^{\lfloor p/3\rfloor}
(-1)^{pn-pm-r}\binom{3pm+3r}{pm+r}\binom{2pm+2r}{pm+r}
\binom{pn}{3pm+3r}\binom{pn+pm+r}{pm+r}3^{pn-3pm-3r}.$$

\noindent
Let's isolate the case $r=0$ and subtract $a_0(n)$ from it, then implement the milder form $\binom{pb}{pd}\equiv_p\binom{b}d$ and bring Corollary 2.7 to bear. The outcome is:
\begin{align*} &\sum_{m=0}^{\lfloor n/3\rfloor}
(-1)^{n-m}\binom{3m}m\binom{2m}m
\binom{n}{3m}\binom{n+m}m3^{pn-3pm}-a_0(n) \\
&=\sum_{m=0}^{\lfloor n/3\rfloor}
(-1)^{n-m}\binom{3m}m\binom{2m}m
\binom{n}{3m}\binom{n+m}m3^{n-3m}\left(3^{(p-1)(n-3m)}-1\right) \\
&\equiv_{p^2}p\sum_{m=0}^{\lfloor n/3\rfloor}(-1)^{n-m}\binom{3m}{m,m,m}\binom{n}{3m}
\binom{n+m}m\frac{(n-3m)}{3^{-n+3m+1}}\sum_{k=1}^{\lfloor p/3\rfloor}\binom{3k}{k,k,k}\frac{3^{-3k}}k.
\end{align*}
The remaining sum in $a_0(pn)$, without the term $r=0$ and modulo $p^2$, takes the form
\begin{align} \sum_{m=0}^{\lfloor n/3\rfloor}\sum_{r=1}^{\lfloor p/3\rfloor}
(-1)^{n-m-r}&\binom{3m}m\binom{3r}r\binom{2m}m\binom{2r}r
\binom{pn}{3pm+3r}\binom{n+m}m3^{n-3m-3r}. \end{align}
To this sum, we apply Gessel's congruence $\binom{p-1}i\equiv_{p^2}(-1)^i$,  with $0<i<p$, so that
\begin{align*} \binom{pn}{3pm+3r}&=\frac{pn}{3pm+3r}\binom{pn-1}{3pm+3r-1} \\
&=\frac{pn}{3pm+3r}\binom{p(n-1)+p-1}{3pm+3r-1}\equiv_{p^2}(-1)^{r-1}\frac{pn}{3r}\binom{n-1}{3m},
\end{align*}
which converts (4.1), modulo $p^2$, into 
\begin{align*} &pn\sum_{m=0}^{\lfloor n/3\rfloor}\sum_{r=1}^{\lfloor p/3\rfloor}
(-1)^{n-m-1}\binom{3m}m\binom{3r}r\binom{2m}m\binom{2r}r
\frac1{3r}\binom{n-1}{3m}\binom{n+m}m3^{n-3m-3r} \\
&\equiv_{p^2} -p\sum_{m=0}^{\lfloor n/3\rfloor}(-1)^{n-m}\binom{3m}{m,m,m}\binom{n}{3m}\binom{n+m}m\frac{(n-3m)}{3^{-n+3m+1}}
\sum_{r=1}^{\lfloor p/3\rfloor}\binom{3r}r\binom{2r}r\frac{3^{-3r}}r.
\end{align*}
However, the two sums (one for $r=0$ above, the other for $r>0$) add up to zero. That means $a_0(pn)-a_0(n)\equiv_{p^2}0$, as desired.
\end{proof}

\bigskip

\smallskip
\noindent
\textbf{Acknowledgements.} The author is grateful to Armin Straub for bringing the conjecture on the Almkvist-Zudilin numbers to his attention.

\enddocument